\documentclass{siamart0516}
  \usepackage{hyperref}
\usepackage{latexsym,amssymb,enumerate,amsmath} 
  \usepackage{graphicx} 
  \usepackage{tikz}
  \usepackage[square,sort,comma,numbers]{natbib}
\usepackage{pgfplots}

\numberwithin{equation}{section}
\theoremstyle{definition}
\newtheorem*{definition2*}{Definition}
\newtheorem{remark}{Remark}

\newcommand{\dd}{\textup{d}}
\def\eps{\varepsilon}
\def\E{\mathbb{E}}
\def\P{\mathbb{P}}
\def\R{\mathbb{R}}

\def\black{\color{black}}

\def\s{\sigma}
\def\L{\mathcal{L}}
\newcommand{\deuc}{L_{\textup{euc}}}
\newcommand{\drie}{L_{\textup{rie}}}
\newcommand{\phys}{\textup{phys}}
\newcommand{\no}{\textup{empty}}

\usepackage{amsfonts}
\usepackage{graphicx}
\usepackage{epstopdf}
\usepackage{algorithmic}
\ifpdf
  \DeclareGraphicsExtensions{.eps,.pdf,.png,.jpg}
\else
  \DeclareGraphicsExtensions{.eps}
\fi

\newcommand{\TheTitle}{The effects of fast inactivation on conditional first passage times of mortal diffusive searchers}
\newcommand{\ShortTitle}{Fast inactivation and first passage times}
\newcommand{\TheAuthors}{Sean D. Lawley}

\headers{\ShortTitle}{\TheAuthors}

\title{{\TheTitle}\thanks{
\funding{The author was supported by the National Science Foundation (DMS-1944574, DMS-1814832, and DMS-1148230).}}}

\author{Sean D. Lawley\thanks{Department of Mathematics, University of Utah, Salt Lake City, UT 84112 USA (\texttt{lawley@math.utah.edu}).}}
\date{\today}

\ifpdf
\hypersetup{
  pdftitle={\TheTitle},
  pdfauthor={\TheAuthors}
}
\fi

\begin{document}

\maketitle


\begin{abstract}
The first time a searcher finds a target is called a first passage time (FPT). In many physical, chemical, and biological processes, the searcher is ``mortal,'' which means that the searcher might become inactivated (degrade, die, etc.)\  before finding the target. In the context of intracellular signaling, an important recent work discovered that fast inactivation can drastically alter the conditional FPT distribution of a mortal diffusive searcher, if the searcher is conditioned to find the target before inactivation. In this paper, we prove a general theorem which yields an explicit formula for all the moments of such conditional FPTs in the fast inactivation limit. This formula is quite universal, as it holds under very general conditions on the diffusive searcher dynamics, the target, and the spatial domain. These results prove in significant generality that if inactivation is fast, then the conditional FPT compared to the FPT without inactivation is (i) much faster, (ii) much less affected by spatial heterogeneity, and (iii) much less variable. Our results agree with recent computational and theoretical analysis of a certain discrete intracellular diffusion model and confirm a conjecture related to the effect of spatial heterogeneity on intracellular signaling. 
\end{abstract}



\section{Introduction}
 
A first passage time (FPT) is the first time a ``searcher'' finds a ``target.'' FPTs have been used extensively to understand timescales in a vast array of physical, chemical, and biological systems \cite{redner2001}. Intracellular signaling processes provide prototypical examples, as signal propagation can depend on a protein (the searcher) diffusing from the cell membrane to the nucleus (the target) \cite{liu2018}. Indeed, questions in cell biology have been especially important in driving FPT research in the past few decades. For example, the complexity of cellular systems has prompted the study of how FPTs depend on small targets \cite{holcman2014, schuss_narrow_2007, ward10, ward10b}, heterogeneous diffusion \cite{chubynsky2014, godec2017, vaccario2015, PB7, PB8, PB10}, evacuation processes \cite{newby2016}, {\black the initial distance between searchers and targets \cite{kolesov2007, pulkkinen2013},} stochastically gated targets \cite{Reingruber2009, reingruber2010, Ammari2011, PB2, PB3, PB11}, moving targets \cite{lindsay2015, lawley2019dtmfpt}, and reversible binding \cite{grebenkov2017imp, lawley2019imp}.

An especially challenging issue in studying intracellular signaling pathways stems from the complicated geometry of the cytosolic space \cite{goodsell2018}. While mathematical models often depict the cytoplasm as an empty space \cite{munoz2009, neves2008, giese2018}, it is actually crowded, tortuous, and highly heterogeneous \cite{blum1989, golding2006, etoc2018}. In fact, several important works have found that such tortuous and crowded geometries can drastically affect FPTs \cite{benichou2010geometry, isaacson2011, woringer2014}. More precisely, let $\tau_{\phys}>0$ denote the FPT in the physiological case in which an intracellular searcher diffuses through a crowded, heterogeneous space to reach the target (see Figure~\ref{figschem}a). Further, let $\tau_{\no}>0$ denote the FPT for the same searcher, except that it diffuses in an empty, homogeneous space to reach the target (see Figure~\ref{figschem}b). These prior works found that $\tau_{\phys}$ and $\tau_{\no}$ can have vastly different statistics \cite{benichou2010geometry, isaacson2011, woringer2014}.

\begin{figure}
  \centering
    \includegraphics[width=1\textwidth]{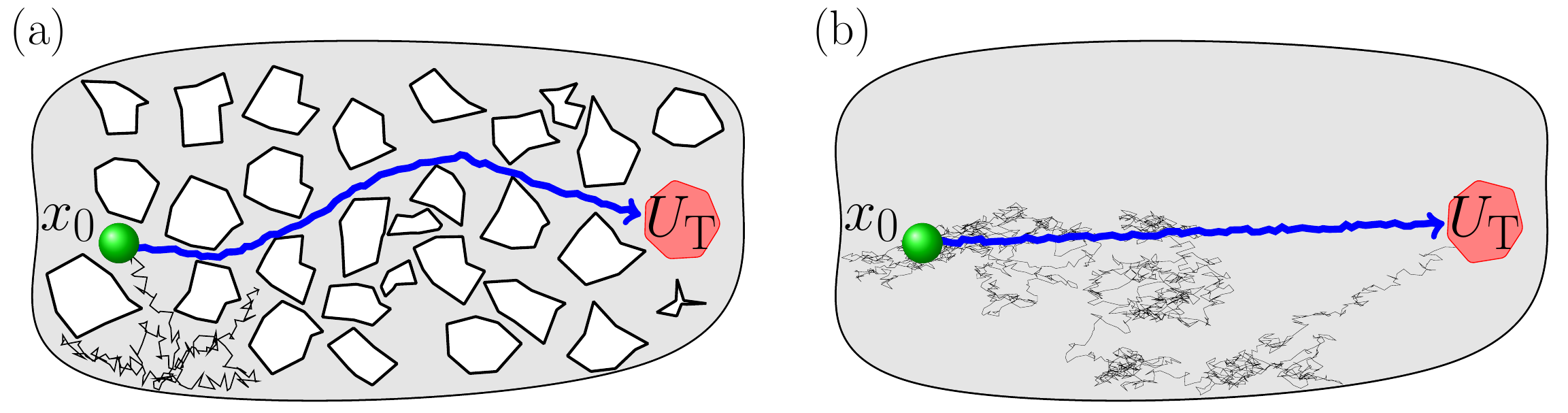}
 \caption{A crowded domain with many reflecting obstacles (panel (a)) and an idealized empty domain (panel (b)). In each panel, the green ball labeled $x_{0}$ is the searcher starting position and the red region labeled $U_{\textup{T}}$ is the target. In each panel, the thin black trajectories depict typical searcher paths, which yield vastly different FPTs to the target ($\tau_{\phys}$ in (a) and $\tau_{\no}$ in (b)), as a searcher can get stuck in the maze of obstacles in (a). The thick blue trajectories correspond to searchers which reach the target very quickly by following the shortest path to the target. The lengths of the shortest paths ($L_{\phys}$ in (a) and $L_{\no}$ in (b)) and the corresponding searcher FPTs in the two panels are similar. Searchers which deviate from a direct path to the target can be filtered out by either considering the fastest FPT out of many FPTs \eqref{fastest} or by conditioning that the searcher reaches the target before a fast inactivation time.
 }
 \label{figschem}
\end{figure}

It was recently posited that the effects of a crowded geometry can be mitigated if the system depends on the arrival of the fastest searcher out of many searchers, since the fastest searcher tends to follow the shortest path to the target \cite{lawley2020uni}. To explain more precisely, let $\{\tau_{\phys,n}\}_{n=1}^{N}$ be the FPTs of $N\gg1$ independent and identically distributed (iid) random searchers diffusing through a crowded space with obstacles. Similarly, let $\{\tau_{\no,n}\}_{n=1}^{N}$ be the FPTs of $N\gg1$ iid random searchers diffusing through an empty space. Define the fastest FPTs (also called \emph{extreme} FPTs \cite{lawley2020uni, lawley2020esp1, lawley2020esp4, basnayake2019, schuss2019, coombs2019, redner2019, sokolov2019, rusakov2019, martyushev2019, tamm2019, basnayake2019c}),
\begin{align}\label{fastest}
\begin{split}
T_{\phys,N}
&:=\min\{\tau_{\phys,1},\dots,\tau_{\phys,N}\},\\
T_{\no,N}
&:=\min\{\tau_{\no,1},\dots,\tau_{\no,N}\}.
\end{split}
\end{align}
Under some mild assumptions, the following asymptotic behavior of the $m$th moment of the fastest FPTs was proven in the case of many searchers \cite{lawley2020uni},
\begin{align}\label{previous}
\E[(T_{\phys,N})^{m}]\sim\Big(\frac{L_{\phys}^{2}}{4D\ln N}\Big)^{m},\quad
\E[(T_{\no,N})^{m}]\sim\Big(\frac{L_{\no}^{2}}{4D\ln N}\Big)^{m},\quad\text{as }N\to\infty,
\end{align}
where $D>0$ is the searcher diffusivity, $L_{\phys}>0$ is the length of the shortest path from the possible searcher starting locations to the target that avoids obstacles and $L_{\no}>0$ is the length of the shortest path from the possible searcher starting locations to the target regardless of obstacles. 
Even in very tortuous domains, it is reasonable to expect that $L_{\phys}\approx L_{\no}$, and thus \eqref{previous} implies that $T_{\phys,N}$ and $T_{\no,N}$ have similar statistics if there are many searchers (even though $\tau_{\phys}$ and $\tau_{\no}$ have vastly different statistics). This is illustrated in Figure~\ref{figschem}.

In addition, the results of \cite{lawley2020uni} imply that fastest FPTs decrease variability. Indeed, it follows from \eqref{previous} that the coefficient of variation of the fastest FPT vanishes as the number of searchers grows,
\begin{align*}
\frac{\sqrt{\text{Variance}[T_{N}]}}{\E[T_{N}]}
\to0\quad\text{as }N\to\infty,
\end{align*}
where $T_{N}$ denotes $T_{\phys,N}$ or $T_{\no,N}$. Summarizing, if there are many searchers, then the fastest FPT compared to a single FPT is (i) much faster, (ii) much less affected by obstacles, and (iii) much less variable. These three points ultimately stem from the fact that considering the fastest FPT filters out searchers which do not take {\black the shortest path to the target. In the language of \cite{godec2016x, grebenkov2018strong}, the fastest searchers take ``direct paths'' rather than ``indirect paths.''}

An important recent work \cite{ma2020} revealed that \emph{fast inactivation} can have a similar effect on FPTs by filtering out searchers which deviate from the shortest path to the target. In \cite{ma2020}, the authors reconstructed organelle surfaces within human B cells from soft X-ray tomographic images and modeled them as reflecting obstacles in the cytoplasm. They then numerically solved the Fokker-Planck equation describing the distribution of a diffusive searcher moving through this heterogeneous domain. From this distribution, they calculated the full distribution of the FPT for the searcher to reach the target (the nucleus). They also derived several interesting rigorous mathematical results regarding their discretized diffusion model.

In the case that the searcher can diffuse indefinitely before finding the target (i.e.\ no inactivation), these authors found that the presence of reflecting organelle barriers significantly affects FPTs \cite{ma2020}. More precisely, if $\tau_{\phys}$ denotes the FPT for the physiological case which includes the organelle barriers and $\tau_{\no}$ denotes the FPT in the case of an empty cytoplasm, then they found that $\tau_{\phys}$ is typically much slower and much more variable than $\tau_{\no}$.

However, since signaling molecules cannot actually diffuse indefinitely, an exponentially distributed \emph{inactivation time} was introduced in \cite{ma2020}. Remarkably, these authors found that if one conditions that the searcher reaches the target before the inactivation time, then the FPT statistics can be dramatically altered. In particular, if inactivation is fast, then the conditional FPT compared to the FPT without inactivation is (i) much faster, (ii) much less affected by obstacles, and (iii) much less variable \cite{ma2020}.

More generally, the finite lifetime of a random searcher characterizes many biological and chemical processes. Indeed, the competition between search and inactivation has been studied in the physics and chemistry literatures, where such searchers are called ``mortal'' or ``evanescent'' \cite{abad2010, abad2012, abad2013, yuste2013, meerson2015b, meerson2015, grebenkov2017}. Motivations for this prior work include (1) sperm cells searching for an egg despite a very high mortality rate, (2) animals or bacteria foraging for food, (3) extinction of a fluorescent signal in certain bio-imaging methods, (4) messenger RNA searching for a ribosome, and (5) nuclear waste storage \cite{grebenkov2017}.

In this paper, we investigate how fast inactivation affects FPTs in a general mathematical setting. We prove a theorem which gives an explicit asymptotic formula for every moment of a FPT that is conditioned on being less than a fast inactivation time. These moments are given in terms of the short time behavior of the unconditioned FPT distribution on a logarithmic scale. We then combine this theorem with recent results developed for fastest FPTs to obtain a simple and remarkably universal formula for these conditional FPT moments. This formula involves a certain geodesic distance from the possible searcher starting locations to the target. As a corollary to these results, we confirm a conjecture made in \cite{ma2020} regarding how intracellular obstacles affect conditional FPT statistics.

The rest of this paper is organized as follows. In section~\ref{results}, we briefly summarize our main results. In section~\ref{math}, we state and prove our main theorem. In section~\ref{examples}, be apply this theorem to several examples of FPTs for diffusive searchers. We conclude by discussing relations to prior work and biological implications.

\section{Summary of main results}\label{results}

Let $\tau$ be any random variable satisfying
\begin{align}\label{log0}
\lim_{t\to0+}t\ln\P(\textcolor{black}{\tau\le t})=-C<0,
\end{align}
for some constant $C>0$. Suppose ${{\s}}>0$ is an independent gamma random variable with rate $\lambda>0$ and shape ${\beta}>0$. Note that if we take ${\beta}=1$, then $\s$ reduces to an exponential random variable with rate $\lambda>0$. Our main theorem (Theorem~\ref{main})  states that under these assumptions, the $m$th moment of $\tau$ for any $m\ge1$, conditioned that $\tau<\s$, has the following asymptotic behavior, 
\begin{align}\label{res0}
\E[\tau^{m}\,|\,\tau<{{\s}}]
\sim\left(\frac{C}{\lambda}\right)^{m/2}\quad\text{as }\lambda\to\infty.
\end{align}
Throughout this work, ``$f\sim g$'' means $f/g\to1$. As a corollary to Theorem~\ref{main}, we prove that the conditional coefficient of variation vanishes as $\lambda\to\infty$,
\begin{align}\label{cor0}
\frac{\sqrt{\text{Variance}[\tau\,|\,\tau<\sigma]}}{\E[\tau\,|\,\tau<\s]}
\to0\quad\text{as }\lambda\to\infty.
\end{align}

The proof of \eqref{res0}-\eqref{cor0} makes no reference to FPTs. That is, \eqref{res0}-\eqref{cor0} are general statements that hold for any random variable $\tau$ satisfying \eqref{log0}. However, we have in mind that $\tau$ is the FPT of a searcher to find a target, and $\s$ is the inactivation time. Hence, \eqref{res0} gives the statistics of the FPT in the event that the searcher finds the target before it is inactivated, in the limit of fast inactivation (large $\lambda$). Notice that \eqref{res0} is independent of the shape parameter $\beta>0$ of the inactivation time (see Remark~\ref{beta} for more on this point).

In view of \eqref{log0} and \eqref{res0}, two questions immediately arise.
\begin{enumerate}[(a)]
\item
What FPTs satisfy \eqref{log0}?
\item
If a FPT satisfies \eqref{log0}, what is $C$?
\end{enumerate}
We answer both of these questions in section~\ref{examples} for several very general examples. Indeed, these examples include $d$-dimensional diffusion processes (i) with general space-dependent diffusion coefficients and drift fields, (ii) on Riemannian manifolds, (iii) with reflecting obstacles, and (iv) with partially absorbing targets.

We find that (a) FPTs of diffusive searchers seem to universally satisfy \eqref{log0} {\black as long as the searchers cannot start arbitrarily close to the target (see the Discussion section for the case that searchers can start arbitrarily close to the target). Further, we find that }(b) the constant $C$ is
\begin{align}\label{C0}
C=\frac{L^{2}}{4D}>0,
\end{align}
where $D>0$ is a characteristic diffusion coefficient and $L>0$ is a certain geodesic distance. Hence, combining \eqref{C0} with \eqref{res0} yields
\begin{align}\label{formula0}
\E[\tau^{m}\,|\,\tau<\s]
\sim\left(\frac{L}{2\sqrt{D\lambda}}\right)^{m}\quad\text{as }\lambda\to\infty.
\end{align}
The geodesic distance $L>0$ is given precisely in section~\ref{examples}. However, we note here that $L>0$ is a certain geodesic between the possible searcher starting locations and the target that (i) avoids reflecting obstacles, (ii) includes any spatial variation or anisotropy in diffusivity, and (iii) includes any geometry in the case of diffusion on a curved space. In addition, the length $L$ is unaffected by forces (i.e.\ a drift field) on the diffusive searcher or a finite absorption rate at the target. The result in \eqref{C0} relies on recent results developed for fastest FPTs \cite{lawley2020uni}.

\section{General analysis of a conditional random variable}\label{math}

In this section, we prove our main theorem. We begin by collecting some basic definitions.


\begin{definition2*}
A random variable ${\s}>0$ is a \textbf{gamma random variable with rate $\lambda>0$ and shape ${\beta}>0$} if its survival probability and probability density function are respectively
\begin{align}\label{gamma}
\begin{split}
S_{{\s}}(t)
&:=\P({\s}>t)
=\frac{\Gamma({\beta},\lambda t)}{\Gamma({\beta})},\quad t\ge0,\\
f_{{\s}}(t)
&:=-\frac{\dd}{\dd t}S_{{\s}}(t)
=\frac{\lambda^{{\beta}}}{\Gamma({\beta})}t^{{\beta}-1}e^{-\lambda t},\quad t>0,
\end{split}
\end{align}
where $\Gamma({\beta},\lambda t)$ and $\Gamma({\beta})$ denote the incomplete and complete gamma functions,
\begin{align*}
\Gamma({\beta},\Lambda)
:=\int_{\Lambda}^{\infty}x^{{\beta}-1}e^{-x}\,\dd x,\quad
\Gamma({\beta})
:=\int_{0}^{\infty}x^{{\beta}-1}e^{-x}\,\dd x.
\end{align*}
If ${\s}$ is a gamma random variable with rate $\lambda>0$ and shape ${\beta}=1$, then \textbf{${\s}$ is an exponential random variable with rate $\lambda>0$}, and \eqref{gamma} simplifies to
\begin{align*}
S_{{\s}}(t)
&:=\P({\s}>t)
=e^{-\lambda t},\quad t\ge0,\\
f_{{\s}}(t)
&:=-\frac{\dd}{\dd t}S_{{\s}}(t)
=\lambda e^{-\lambda t},\quad t>0.
\end{align*}
\end{definition2*}


The next result is our main theorem. It gives an asymptotic formula for the conditional moments of a random variable based on the behavior of its unconditioned distribution on a logarithmic scale. Throughout this work, ``$f\sim g$'' means $f/g\to1$.

\begin{theorem}\label{main}
Let $\tau$ be any random variable satisfying
\begin{align}\label{log}
\lim_{t\to0+}t\ln\P(\textcolor{black}{\tau\le t})=-C<0,
\end{align}
for some constant $C>0$. Suppose ${{\s}}>0$ is an independent gamma random variable with rate $\lambda>0$ and shape ${\beta}>0$ (if ${\beta}=1$, then $\s$ is an exponential random variable). If $m\ge1$, then
\begin{align}\label{res}
\E[\tau^{m}\,|\,\tau<{{\s}}]
:=\frac{\E[\tau^{m}1_{\tau<{{\s}}}]}{\P(\tau<{{\s}})}
\sim\left(\frac{C}{\lambda}\right)^{m/2}\quad\text{as }\lambda\to\infty.
\end{align}
\end{theorem}

Note that \eqref{log} ensures that $\tau$ is positive with probability one. Note also that the equality in \eqref{res} is merely the definition of conditional expectation, where $1_{\tau<\s}$ denotes the indicator function on the event $\tau<\s$. That is,
\begin{align*}
1_{\tau<\s}
=\begin{cases}
1 & \text{if }\tau<\s,\\
0 & \text{if }\tau\ge\s.
\end{cases}
\end{align*}

\begin{remark}\label{beta}{\rm
Notice that \eqref{res} is independent of the shape parameter $\beta$ of the inactivation time $\sigma$. Thus, the conditional FPT moments for (i) an exponentially distributed inactivation time ($\beta=1$) or (ii) any gamma distributed inactivation time (any $\beta>0$) are identical for fast inactivation. This is counterintuitive since, for example, the mean of $\sigma$ grows linearly with $\beta$. Moreover, if $\beta$ is an integer, then $\s$ is equal in distribution to the sum of $\beta$-many independent, exponentially distributed random variables, each with rate $\lambda>0$. Therefore, Theorem~\ref{main} implies that the FPT conditioned on being less than a single exponential time with rate $\lambda$ is identical to the FPT conditioned on being less than a sum of any fixed number of exponential times with rate $\lambda$, provided $\lambda\gg1$. That is, if $\sigma_{1},\dots,\sigma_{\beta}$ are iid with $\sigma_{k}$ exponentially distributed with rate $\lambda>0$ for each $k\in\{1,\dots,\beta\}$, then
\begin{align*}
\E[\tau^{m}\,|\,\tau<\sigma_{1}]
\sim\E\Big[\tau^{m}\,|\,\tau<\sum_{k=1}^{\beta}\sigma_{k}\Big]\quad\text{as }\lambda\to\infty.
\end{align*}
}\end{remark}

The following result is an immediate corollary of Theorem~\ref{main}.
\begin{corollary}\label{cor}
Let $\tau$ and $\s$ be as in Theorem~\ref{main}. The variance of $\tau$, conditioned that $\tau<\s$, which is defined as
\begin{align*}
\textup{Variance}[\tau\,|\,\tau<\s]
:=\E\Big[\big(\tau-\E[\tau\,|\,\tau<\s]\big)^{2}\,|\,\tau<\s\Big],
\end{align*}
vanishes faster than $1/\lambda$ as $\lambda\to\infty$. That is,
\begin{align*}
\lambda\textup{Variance}[\tau\,|\,\tau<\s]
\to0\quad\text{as }\lambda\to\infty.
\end{align*}
Moreover, the coefficient of variation of $\tau$, conditioned that $\tau<\s$, vanishes as $\lambda\to\infty$. That is,
\begin{align*}
\frac{\sqrt{\textup{Variance}[\tau\,|\,\tau<\s]}}{\E[\tau\,|\,\tau<\s]}
\to0\quad\text{as }\lambda\to\infty.
\end{align*}
\end{corollary}

In order to prove Theorem~\ref{main}, the next lemma puts the conditional $m$th moment, $\E[\tau^{m}\,|\,\tau<{\s}]$, in a form that is convenient for analysis. 

\begin{lemma}\label{convenient}
Let $\tau>0$ be a positive random variable with cumulative distribution
\begin{align*}
F(t):=\P(\tau\le t).
\end{align*}
Let ${{\s}}>0$ be an independent positive random variable with survival probability and probability density function denoted respectively by
\begin{align*}
S_{{\s}}(t):=\P({\s}>t),\quad
f_{{\s}}(t):=-\tfrac{\dd}{\dd t}S_{{\s}}(t).
\end{align*}
If $\E[{\s}^{m}]<\infty$ for some $m>0$, then  
\begin{align}\label{ratio0}
\E[\tau^{m}\,|\,\tau<{\s}]
&=\frac{\int_{0}^{\infty} t^{m} f_{{\s}}(t)F(t)\,\dd t}{\int_{0}^{\infty}f_{{\s}}(t)F(t) \,\dd t}
-\frac{\int_{0}^{\infty}mt^{m-1} S_{{\s}}(t)F(t)\,\dd t}{\int_{0}^{\infty}f_{{\s}}(t)F(t)\,\dd t}.
\end{align}
\end{lemma}


We prove Theorem~\ref{main}, Corollary~\ref{cor}, and Lemma~\ref{convenient} in the subsections below. As the full proof of Theorem~\ref{main} is relatively long, we end this subsection by sketching the main idea of the proof.

\begin{proof}[Sketch of proof of Theorem~\ref{main}]

Upon using the formulas for $f_{{\s}}(t)$ and $S_{{\s}}(t)$ in \eqref{gamma}, the expression in \eqref{ratio0} becomes
\begin{align}\label{ratios}
\E[\tau^{m}\,|\,\tau<{\s}]
&=\frac{\int_{0}^{\infty} t^{m+{\beta}-1} e^{-\lambda t}F(t)\,\dd t}{\int_{0}^{\infty}t^{{\beta}-1} e^{-\lambda t}F(t) \,\dd t}
-\frac{\int_{0}^{\infty}mt^{m-1} \Gamma({\beta},\lambda t)F(t)\,\dd t}{\lambda^{{\beta}}\int_{0}^{\infty}t^{{\beta}-1} e^{-\lambda t}F(t) \,\dd t}.
\end{align}
It is clear that the large $\lambda$ behavior of the integrals in \eqref{ratios} is determined by the short time behavior of the integrands. In particular, for any power $p>-1$, we expect that
\begin{align*}
\int_{0}^{\infty}t^{p} e^{-\lambda t}F(t) \,\dd t
&\approx\int_{0}^{\eps}t^{p} e^{-\lambda t}F(t) \,\dd t\quad\text{for $\eps>0$ and $\lambda\gg1$}.
\end{align*}
Furthermore, the assumption in \eqref{log} means very roughly that
\begin{align*}
F(t)
\approx e^{-C/t}\quad\text{for }t\ll1.
\end{align*}
Hence, we expect that 
\begin{align}\label{simple}
\int_{0}^{\eps}t^{p} e^{-\lambda t}F(t) \,\dd t
&\approx\int_{0}^{\eps}t^{p} e^{-\lambda t-C/t} \,\dd t\quad\text{for $\eps>0$ and $\lambda\gg1$}.
\end{align}
To determine the large $\lambda$ behavior of the integral in the righthand side of \eqref{simple}, we first note that a simple calculus exercise shows that the maximum of the exponential factor in the integrand occurs at $t=\sqrt{C/\lambda}$. We therefore make the change of variables,
\begin{align*}
s=t\sqrt{\lambda/C},
\end{align*}
and obtain
\begin{align}
\int_{0}^{\eps} t^{p}e^{-\lambda t-C/t}\,\dd t
=\left(\frac{C}{\lambda}\right)^{(p+1)/2}\int_{0}^{\sqrt{\lambda/C}\eps}s^{p}e^{-\sqrt{\lambda C}(s+1/s)}\,\dd s,\label{simple2}
\end{align}
and the maximum of the exponential factor in the integrand now occurs at $s=1$. Applying Laplace's method \cite{bender2013} to \eqref{simple2} then yields
\begin{align*}
\int_{0}^{\sqrt{\lambda/C}\eps}s^{p}e^{-\sqrt{\lambda C}(s+1/s)}\,\dd s
\sim
\sqrt{\frac{2\pi}{2\sqrt{\lambda C}}}e^{-2\sqrt{\lambda/C}}\quad\text{as }\lambda\to\infty.
\end{align*}
If we approximate the incomplete gamma function as
\begin{align*}
\Gamma({\beta},\lambda t)\approx(\lambda t)^{{\beta}-1}e^{-\lambda t},\quad\text{for }\lambda\gg1,
\end{align*}
then we similarly obtain
\begin{align*}
\begin{split}
\int_{0}^{\infty}mt^{m-1} \Gamma({\beta},\lambda t)F(t)\,\dd t
&\approx\lambda^{{\beta}-1}m\int_{0}^{\infty}t^{m+{\beta}-2} e^{-\lambda t}F(t)\,\dd t\\
&\approx\lambda^{{\beta}-1}m\left(\frac{C}{\lambda}\right)^{\frac{m+{\beta}-1}{2}}\sqrt{\frac{2\pi}{2\sqrt{\lambda C}}}e^{-2\sqrt{\lambda/C}}\quad\text{for }\lambda\gg1.
\end{split}
\end{align*}
Combining these approximations with \eqref{ratios} yields the desired result,
\begin{align*}
\E[\tau^{m}\,|\,\tau<{\s}]
&\approx
\frac{(\frac{C}{\lambda})^{\frac{m+{\beta}}{2}}\sqrt{\frac{2\pi}{2\sqrt{\lambda C}}}e^{-2\sqrt{\lambda/C}}}{(\frac{C}{\lambda})^{\frac{{\beta}}{2}}\sqrt{\frac{2\pi}{2\sqrt{\lambda C}}}e^{-2\sqrt{\lambda/C}}}
-\frac{\lambda^{{\beta}-1}m(\frac{C}{\lambda})^{\frac{m+{\beta}-1}{2}}\sqrt{\frac{2\pi}{2\sqrt{\lambda C}}}e^{-2\sqrt{\lambda/C}}}{\lambda^{{\beta}}(\frac{C}{\lambda})^{\frac{{\beta}}{2}}\sqrt{\frac{2\pi}{2\sqrt{\lambda C}}}e^{-2\sqrt{\lambda/C}}}\\
&=\left(\frac{C}{\lambda}\right)^{m/2}
-\frac{m}{\lambda}\left(\frac{C}{\lambda}\right)^{\frac{m-1}{2}}
\approx\left(\frac{C}{\lambda}\right)^{m/2}\quad\text{for }\lambda\gg1.
\end{align*}
We make this argument rigorous in the subsections below.
\end{proof}

\subsection{Three lemmas and the proof of Theorem~\ref{main}}

\begin{lemma}\label{pp}
Let $\tau>0$ be as in Theorem~\ref{main} with $F(t):=\P(\textcolor{black}{\tau\le t})$. Define
\begin{align}\label{h}
h(t)
:=e^{C/t}F(t),\quad t>0.
\end{align}
If $p>-1$ and $\eps\in(0,1)$, then
\begin{align}
&\int_{0}^{\infty}t^{p}e^{-\lambda t}F(t)\,\dd t
=\left(\frac{C}{\lambda}\right)^{(p+1)/2}\int_{0}^{\infty}s^{p} e^{-\sqrt{\lambda C}(s+1/s)}h(\sqrt{C/\lambda}s)\,\dd s\label{claimcv}\\
&\quad\sim\left(\frac{C}{\lambda}\right)^{(p+1)/2}\int_{1-\eps}^{1+\eps}s^{p}  e^{-\sqrt{\lambda C}(s+1/s)}h(\sqrt{C/\lambda}s)\,\dd s\quad\text{as }\lambda\to\infty.\label{claim0}
\end{align}
\end{lemma}

\begin{lemma}\label{pq}
If $p>-1$, $q>-1$, and $h$ is as in \eqref{h}, then
\begin{align*}
&\int_{0}^{\infty}s^{p} e^{-\sqrt{\lambda C}(s+1/s)}h(\sqrt{C/\lambda}s)\,\dd s
\sim\int_{0}^{\infty}s^{q}e^{-\sqrt{\lambda C}(s+1/s)}h(\sqrt{C/\lambda}s)\,\dd s\quad\text{as }\lambda\to\infty.
\end{align*}
\end{lemma}

\begin{lemma}\label{ps}
Under the assumptions of Theorem~\ref{main}, we have that
\begin{align*}
\lim_{\lambda\to\infty}\lambda^{m/2}\frac{\int_{0}^{\infty}mt^{m-1} S_{{\s}}(t)F(t)\,\dd t}{\int_{0}^{\infty}f_{{\s}}(t)F(t)\,\dd t}
=0.
\end{align*}
\end{lemma}


\begin{proof}[Proof of Theorem~\ref{main}]
By Lemma~\ref{convenient} and Lemma~\ref{ps}, we have that
\begin{align}\label{dm}
\begin{split}
&\lim_{\lambda\to\infty}\Big(\frac{\lambda}{C}\Big)^{m/2}\E[\tau^{m}\,|\,\tau<{\s}]
=\lim_{\lambda\to\infty}\Big(\frac{\lambda}{C}\Big)^{m/2}\frac{\int_{0}^{\infty} t^{m} f_{{\s}}(t)F(t)\,\dd t}{\int_{0}^{\infty}f_{{\s}}(t)F(t) \,\dd t}.
\end{split}
\end{align}
By \eqref{gamma} and Lemmas~\ref{pp} and \ref{pq}, we have that
\begin{align}\label{dom}
\begin{split}
&\lim_{\lambda\to\infty}\Big(\frac{\lambda}{C}\Big)^{\frac{m}{2}}\frac{\int_{0}^{\infty} t^{m} f_{{\s}}(t)F(t)\,\dd t}{\int_{0}^{\infty}f_{{\s}}(t)F(t) \,\dd t}
=\lim_{\lambda\to\infty}\Big(\frac{\lambda}{C}\Big)^{\frac{m}{2}}\frac{\int_{0}^{\infty} t^{m+{\beta}-1} e^{-\lambda t}F(t)\,\dd t}{\int_{0}^{\infty}t^{{\beta}-1}e^{-\lambda t}F(t) \,\dd t}
=1.
\end{split}
\end{align}
Combining \eqref{dm} and \eqref{dom} completes the proof.
\end{proof}

\subsection{Proofs of Corollary~\ref{cor} and Lemmas~\ref{convenient}-\ref{ps}}

We begin this subsection by proving Corollary~\ref{cor}.

\begin{proof}[Proof of Corollary~\ref{cor}]
By Theorem~\ref{main}, we have that
\begin{align}\label{varconv}
\begin{split}
\lambda\textup{Variance}[\tau\,|\,\tau<\s]
&=\lambda\big(\E[\tau^{2}\,|\,\tau<\s]-(\E[\tau\,|\,\tau<\s])^{2}\big)\\
&=C\frac{\E[\tau^{2}\,|\,\tau<\s]}{\frac{C}{\lambda}}-C\frac{(\E[\tau\,|\,\tau<\s])^{2}}{\frac{C}{\lambda}}
\to0\quad\text{as }\lambda\to\infty.
\end{split}
\end{align}
By Theorem~\ref{main}, 
\begin{align}\label{ee4}
\frac{\E[\tau\,|\,\tau<\s]}{\sqrt{C/\lambda}}\ge\frac{1}{2},\quad\text{for all sufficiently large $\lambda$}.
\end{align}
Let $\eps>0$. By \eqref{varconv}, we have that
\begin{align}\label{ee5}
\sqrt{\lambda}\sqrt{\textup{Variance}[\tau\,|\,\tau<\s]}<\eps,\quad\text{for all sufficiently large $\lambda$}.
\end{align}
Therefore, combining \eqref{ee4} and \eqref{ee5} gives
\begin{align*}
\frac{\sqrt{\textup{Variance}[\tau\,|\,\tau<\s]}}{\E[\tau\,|\,\tau<\s]}
\le\frac{2\eps}{\sqrt{C}},\quad\text{for all sufficiently large $\lambda$}.
\end{align*}
Since $\eps>0$ is arbitrary, the proof is complete.
\end{proof}


\begin{proof}[Proof of Lemma~\ref{convenient}]
The mean of any nonnegative random variable $Z\ge0$ is
\begin{align*}
\E[Z]
=\int_{0}^{\infty}\P(Z>z)\,\dd z.
\end{align*}
Since (i) $\tau^{m}1_{\tau<{\s}}\ge0$ and (ii) $\tau<{\s}$ if any only if $\tau^{m}<{\s}^{m}$, we therefore have that
\begin{align*}
\E[\tau^{m}1_{\tau<{\s}}]
=\E[\tau^{m}1_{\tau^{m}<{\s}^{m}}]
&=\int_{0}^{\infty}\P(\tau^{m}1_{\tau^{m}<{\s}^{m}}>t)\,\dd t
=\int_{0}^{\infty}\P(t^{1/m}<\tau<{\s})\,\dd t.
\end{align*}
Denote the survival probability of $\tau$ by
\begin{align*}
S(t)
:=\P(\tau>t)
=1-F(t).
\end{align*}
Since $\tau$ and ${\s}$ are independent, we condition on the value of ${\s}$ to obtain
\begin{align}
\E[\tau^{m}1_{\tau<{\s}}]
&=\int_{0}^{\infty}\int_{0}^{\infty}\P(t^{1/m}<\tau<s)f_{\s}(s)\,\dd s\,\dd t\nonumber\\
&=\int_{0}^{\infty}\int_{t^{1/m}}^{\infty}(S(t^{1/m})-S(s))f_{\s}(s)\,\dd s\,\dd t\nonumber\\
&=\int_{0}^{\infty}\int_{t^{1/m}}^{\infty}S(t^{1/m})f_{\s}(s)\,\dd s\,\dd t
-\int_{0}^{\infty}\int_{0}^{s^{m}}S(s)f_{\s}(s)\,\dd t\,\dd s.\label{ind}
\end{align}
Integrating $\int_{t^{1/m}}^{\infty}f_{\s}(s)\,\dd s=S_{\s}(t^{1/m})$ and $\int_{0}^{s^{m}}1\,\dd t=s^{m}$ in \eqref{ind} and changing variables $t=s^{m}$ then yields 
\begin{align*}
\E[\tau^{m}1_{\tau<{\s}}]
&=\int_{0}^{\infty}S(t^{1/m})S_{\s}(t^{1/m})\,\dd t
-\int_{0}^{\infty}s^{m}S(s)f_{\s}(s)\,\dd s\\
&=\int_{0}^{\infty}\big(S_{\s}(t^{1/m})-(1/m) t^{1/m}f_{\s}(t^{1/m})\big)S(t^{1/m})\,\dd t.
\end{align*}
In terms of $F(t)=1-S(t)$, we therefore have that
\begin{align}
\E[\tau^{m}1_{\tau<{\s}}]
&=\int_{0}^{\infty}\big(S_{\s}(t^{1/m})-(1/m) t^{1/m}f_{\s}(t^{1/m})\big)(1-F(t^{1/m}))\,\dd t\nonumber\\
&=\int_{0}^{\infty}(1/m) t^{1/m} f_{\s}(t^{1/m})F(t^{1/m})\,\dd t
-\int_{0}^{\infty}S_{\s}(t^{1/m})F(t^{1/m})\,\dd t\nonumber\\
&=\int_{0}^{\infty}t^{m} f_{\s}(t)F(t)\,\dd t
-\int_{0}^{\infty}mt^{m-1}S_{\s}(t)F(t)\,\dd t,\label{change}
\end{align}
where we have used that 
\begin{align}\label{ibp}
\int_{0}^{\infty}S_{\s}(t^{1/m})\,\dd t
=\frac{1}{m}\int_{0}^{\infty}t^{1/m}f_{\s}(t^{1/m})\,\dd t,
\end{align}
and changed variables $t\to t^{m}$ in \eqref{change}. The equality in \eqref{ibp} can be established by integration by parts or by noting that each side is equal to $\E[{\s}^{m}]$.

Using the independence of $\tau$ and ${\s}$ and conditioning on the value of ${\s}$ yields
\begin{align}\label{L}
\P(\tau<{\s})
=\int_{0}^{\infty}F(t)f_{\s}(t)\,\dd t.
\end{align}
Hence, combining the definition $\E[\tau^{m}\,|\,\tau<{\s}]
:=\frac{\E[\tau^{m}1_{\tau<{\s}}]}{\P(\tau<{\s})}$ with \eqref{change} and \eqref{L} completes the proof.
\end{proof}

\begin{proof}[Proof of Lemma~\ref{pp}]
Using the definition $h(t):=e^{C/t}F(t)$ and changing of variables $s=\sqrt{\lambda/C}t$ yields \eqref{claimcv}.

Fix $\eps>0$. It remains to show \eqref{claim0}. Define
\begin{align*}
g(t)
:=\ln(h(t))=C/t+\ln F(t),\quad
\phi(t)
:=t+1/t.
\end{align*}
Note that $\phi(1)=2$ is the minimum of $\phi(t)$ for all $t>0$ and that $\phi$ is strictly decreasing for $t\in(0,1)$ and strictly increasing for $t>1$. To prove \eqref{claim0}, observe first that
\begin{align}\label{ff}
\begin{split}
&\int_{1+\eps}^{\infty}s^{p}  e^{-\sqrt{\lambda C}(s+1/s)}h(\sqrt{C/\lambda}s)\,\dd s\\
&\quad=e^{-\sqrt{\lambda C}\delta_{1}}\int_{1+\eps}^{13/6}s^{p} e^{-\sqrt{\lambda C}[\phi(s)-\delta_{1}-\frac{1}{\sqrt{\lambda C}}g(\sqrt{C/\lambda}s)]}\,\dd s\\
&\qquad+e^{-\sqrt{\lambda C}\delta_{1}}\int_{13/6}^{\infty}s^{p} e^{-\sqrt{\lambda C}[\phi(s)-\delta_{1}-\frac{1}{\sqrt{\lambda C}}g(\sqrt{C/\lambda}s)]}\,\dd s,
\end{split}
\end{align}
where
\begin{align}\label{delta1}
\delta_{1}
:=\phi(1+\eps/2)=1+\eps/2+\frac{1}{1+\eps/2}\in(2,13/6),
\end{align}
since $\eps\in(0,1)$. To control the righthand side of \eqref{ff}, we bound the factor appearing in the exponent of the integrand. Specifically, to handle the first integral in the righthand side of \eqref{ff}, we claim that we may take $\lambda$ sufficiently large so that
\begin{align}\label{claim}
H_{1}
:=\phi(s)-\delta_{1}-\frac{1}{sC}\big[\sqrt{C/\lambda}sg(\sqrt{C/\lambda}s)\big]
>\nu
>0,\;\text{for all }s\in[1+\eps,13/6],
\end{align}
for some $\nu\in(0,1)$ that depends on $\eps$ but is independent of $\lambda$. To prove \eqref{claim}, note that \eqref{log} implies that $tg(t)\to0$ as $t\to0+$. Hence, for any $\eta>0$, we may take $\lambda$ sufficiently large so that for all $s\in[1+\eps,13/6]$, we have that
\begin{align}
H_{1}
\ge \phi(s)-\delta_{1} -\eta
&\ge\phi(1+\eps)-\delta_{1}-\eta.\label{lb}
\end{align}
Taking $\eta\in(0,\nu)$ ensures that the lower bound in \eqref{lb} is strictly greater than $\nu$ if
\begin{align*}
\nu
=\tfrac{1}{2}(\phi(1+\eps)-\delta_{1})>0,
\end{align*}
which verifies \eqref{claim}. To handle the last integral in \eqref{ff}, note that
\begin{align*}
tg(t)=C+t\ln F(t)\le C,\quad\text{for all }t>0,
\end{align*}
since $F(t)\in(0,1]$ for all $t>0$. Hence,
\begin{align*}
H_{1}
\ge \phi(s)-\delta_{1}-1/s
=s-\delta_{1}
\ge s-13/6.
\end{align*}
Therefore, we conclude from \eqref{ff} that for sufficiently large $\lambda$,
\begin{align*}
&\int_{1+\eps}^{\infty}s^{p}  e^{-\sqrt{\lambda C}(s+1/s)}h(\sqrt{C/\lambda}s)\,\dd s\\
&\quad\le e^{-\sqrt{\lambda C}\delta_{1}}\Big(\int_{1+\eps}^{13/6}s^{p} e^{-\sqrt{\lambda C}\nu}\,\dd s
+\int_{13/6}^{\infty}s^{p} e^{-\sqrt{\lambda C}[s-13/6]}\,\dd s\Big)
\le e^{-\sqrt{\lambda C}\delta_{1}}.
\end{align*}

Next, observe that
\begin{align}\label{ff2}
\begin{split}
&\int_{0}^{1-\eps}s^{p} e^{-\sqrt{\lambda C}(s+1/s)}h(\sqrt{C/\lambda}s)\,\dd s\\
&\quad=e^{-\sqrt{\lambda C}\delta_{2}}\int_{0}^{1-\eps}s^{p}e^{-\sqrt{\lambda C}[\phi(s)-\delta_{2}-\frac{1}{\sqrt{\lambda C}}g(\sqrt{C/\lambda}s)]}\,\dd s,
\end{split}
\end{align}
where
\begin{align}\label{delta2}
\delta_{2}
:=\phi(1-\eps/2)
=1-\eps/2+\frac{1}{1-\eps/2}\in(2,5/2),
\end{align}
since $\eps\in(0,1)$. Similar to \eqref{claim}, we now claim that
\begin{align}\label{claim2}
H_{2}:=\phi(s)-\delta_{2}-\frac{1}{sC}\big[\sqrt{C/\lambda}sg(\sqrt{C/\lambda}s)\big]>0,\quad\text{for all }s\in(0,1-\eps].
\end{align}
Recall that \eqref{log} implies $tg(t)\to0$ as $t\to0+$, and let
\begin{align}\label{eta}
\eta
=\tfrac{1}{2}(\phi(1-\eps)-\delta_{2})(1-\eps)C\in(0,C(1-(1-\eps)^{2})),
\end{align}
which ensures $1-\eps<\sqrt{1-\eta/C}$. Hence, we may take $\lambda$ sufficiently large so that for all $s\in(0,1-\eps]\subset(0,\sqrt{1-\eta/C})$,
\begin{align*}
H_{2}
\ge \phi(s)-\delta_{2}-\eta/(sC)
\ge \phi(1-\eps)-\eta/((1-\eps)C)-\delta_{2}>0,
\end{align*}
by the choice of $\delta_{2}$ in \eqref{delta2}, $\eta$ in \eqref{eta}, and the fact that $\phi(s)-\eta/(sC)$ is strictly decreasing for $s\in(0,\sqrt{1-\eta/C})$. Hence, \eqref{claim2} is verified. Hence, \eqref{ff2} implies
\begin{align*}
\int_{0}^{1-\eps}s^{p} e^{-\sqrt{\lambda C}(s+1/s)}h(\sqrt{C/\lambda}s)\,\dd s
&\le e^{-\sqrt{\lambda C}\delta_{2}}\int_{0}^{1-\eps}s^{p}\,\dd s\quad\text{for sufficiently large }\lambda.
\end{align*}

In order to verify \eqref{claim0}, it remains to show that
\begin{align}\label{claim3}
\lim_{\lambda\to\infty}\frac{e^{-\sqrt{\lambda C}\min\{\delta_{1},\delta_{2}\}}}{\int_{1-\eps}^{1+\eps}s^{p} e^{-\sqrt{\lambda C}(s+1/s)}h(\sqrt{C/\lambda}s)\,\dd s}
=0
\end{align}
First observe that
\begin{align*}
\int_{1-\eps}^{1+\eps}s^{p} e^{-\sqrt{\lambda C}(s+1/s)}h(\sqrt{C/\lambda}s)\,\dd s
&\ge e^{-\sqrt{\lambda C}\delta_{3}}\int_{1-\eps/4}^{1+\eps/4}s^{p} e^{g(\sqrt{C/\lambda}s)}\,\dd s,
\end{align*}
where
\begin{align}\label{delta3}
\delta_{3}
:=\max\{\phi(1+\eps/4),\phi(1-\eps/4)\}.
\end{align}
Recall that $tg(t)\to0$ as $t\to0+$ by \eqref{log}. Hence, if $\eta>0$, then for large $\lambda$,
\begin{align*}
\sqrt{C/\lambda}sg(\sqrt{C/\lambda}s)\ge-\eta\quad\text{for all }s\in[1-\eps/4,1+\eps/4].
\end{align*}
Hence,
\begin{align*}
\int_{1-\eps/4}^{1+\eps/4}s^{p} e^{g(\sqrt{C/\lambda}s)}\,\dd s
&=\int_{1-\eps/4}^{1+\eps/4}s^{p} e^{\sqrt{\lambda/C}s^{-1}[\sqrt{C/\lambda}sg(\sqrt{C/\lambda}s)]}\,\dd s\\
&\ge\int_{1-\eps/4}^{1+\eps/4}s^{p} e^{-\sqrt{\lambda/C}s^{-1}\eta}\,\dd s\\
&\ge e^{-\sqrt{\lambda C}(1-\eps/4)^{-1}\eta/C}\int_{1-\eps/4}^{1+\eps/4}s^{p} \,\dd s.
\end{align*}
Hence, in order to prove \eqref{claim3}, it is sufficient to have that
\begin{align}\label{suff}
\delta_{3}+(1-\eps/4)^{-1}\eta/C<\min\{\delta_{1},\delta_{2}\}.
\end{align}
Therefore, we may take
\begin{align*}
\eta
=\tfrac{1}{2}(\min\{\delta_{1},\delta_{2}\}-\delta_{3})(1-\eps/4)C>0,
\end{align*}
by our choice of $\delta_{1}$, $\delta_{2}$, and $\delta_{3}$ in \eqref{delta1}, \eqref{delta2}, and \eqref{delta3}. Hence, \eqref{claim3} is verified and therefore \eqref{claim0} is verified.
\end{proof}


\begin{proof}[Proof of Lemma~\ref{pq}]
Let $\eps\in(0,1)$. By Lemma~\ref{pp}, we have that
\begin{align}\label{oo}
\begin{split}
&\frac{\int_{0}^{\infty}s^{p} e^{-\sqrt{\lambda C}(s+1/s)}h(\sqrt{C/\lambda}s)\,\dd s}{\int_{0}^{\infty}s^{q} e^{-\sqrt{\lambda C}(s+1/s)}h(\sqrt{C/\lambda}s)\,\dd s}
\sim\frac{\int_{1-\eps}^{1+\eps}s^{p} e^{-\sqrt{\lambda C}(s+1/s)}h(\sqrt{C/\lambda}s)\,\dd s}{\int_{1-\eps}^{1+\eps}s^{q} e^{-\sqrt{\lambda C}(s+1/s)}h(\sqrt{C/\lambda}s)\,\dd s}
\quad\text{as }\lambda\to\infty.
\end{split}
\end{align}
Without loss of generality, assume $p\ge q$. Hence, for any value of $\lambda>0$, we have that
\begin{align*}
&\frac{\int_{1-\eps}^{1+\eps}s^{p} e^{-\sqrt{\lambda C}(s+1/s)}h(\sqrt{C/\lambda}s)\,\dd s}{\int_{1-\eps}^{1+\eps}s^{q} e^{-\sqrt{\lambda C}(s+1/s)}h(\sqrt{C/\lambda}s)\,\dd s}\\
&\quad\le\frac{(1+\eps)^{p-q}\int_{1-\eps}^{1+\eps}s^{q}e^{-\sqrt{\lambda C}(s+1/s)}h(\sqrt{C/\lambda}s)\,\dd s}{\int_{1-\eps}^{1+\eps}s^{q}e^{-\sqrt{\lambda C}(s+1/s)}h(\sqrt{C/\lambda}s)\,\dd s}
=(1+\eps)^{p-q},
\end{align*}
and similarly,
\begin{align*}
&\frac{\int_{1-\eps}^{1+\eps}s^{p} e^{-\sqrt{\lambda C}(s+1/s)}h(\sqrt{C/\lambda}s)\,\dd s}{\int_{1-\eps}^{1+\eps}s^{q} e^{-\sqrt{\lambda C}(s+1/s)}h(\sqrt{C/\lambda}s)\,\dd s}
\ge(1-\eps)^{p-q}.
\end{align*}
Therefore,
\begin{align}\label{supinf}
\begin{split}
&\limsup_{\lambda\to\infty}\frac{\int_{1-\eps}^{1+\eps}s^{p} e^{-\sqrt{\lambda C}(s+1/s)}h(\sqrt{C/\lambda}s)\,\dd s}{\int_{1-\eps}^{1+\eps}s^{q}e^{-\sqrt{\lambda C}(s+1/s)}h(\sqrt{C/\lambda}s)\,\dd s}
\le(1+\eps)^{p-q},\\
&\liminf_{\lambda\to\infty}\frac{\int_{1-\eps}^{1+\eps}s^{p} e^{-\sqrt{\lambda C}(s+1/s)}h(\sqrt{C/\lambda}s)\,\dd s}{\int_{1-\eps}^{1+\eps}s^{q}e^{-\sqrt{\lambda C}(s+1/s)}h(\sqrt{C/\lambda}s)\,\dd s}
\ge(1-\eps)^{p-q}.
\end{split}
\end{align}
Since $\eps\in(0,1)$ was arbitrary, combining \eqref{supinf} with \eqref{oo} completes the proof.
\end{proof}


\begin{proof}[Proof of Lemma~\ref{ps}]
Using \eqref{gamma}, changing variables $s=\sqrt{\lambda/C}t$, and recalling $h(t):=e^{C/t}F(t)$ gives
\begin{align*}
&\int_{0}^{\infty}mt^{m-1} S_{{\s}}(t)F(t)\,\dd t
=\int_{0}^{\infty}mt^{m-1} \frac{\Gamma({\beta},\lambda t)}{\Gamma({\beta})}e^{-C/t}h(t)\,\dd t\\
&\quad=\frac{m}{\Gamma({\beta})}\Big(\frac{C}{\lambda}\Big)^{m/2}\int_{0}^{\infty}s^{m-1} \Gamma({\beta},\sqrt{\lambda C}s)e^{-\sqrt{\lambda C}/s}h(\sqrt{C/\lambda}s)\,\dd s.
\end{align*}
Let $\delta\in(0,1)$. It is straightforward to show that we have the following upper bound on the incomplete gamma function,
\begin{align}\label{ub}
\Gamma({\beta},\Lambda)\le\Lambda^{{\beta}-1+\delta}e^{-\Lambda}\quad\text{for $\Lambda$ sufficiently large}.
\end{align}
Fix $\eps\in(0,1/6)$. It follows from \eqref{ub} that we may take $\lambda$ sufficiently large so that
\begin{align*}
&\int_{\eps}^{\infty}s^{m-1} \Gamma({\beta},\sqrt{\lambda C}s)e^{-\sqrt{\lambda C}/s}h(\sqrt{C/\lambda}s)\,\dd s\\
&\quad\le(\sqrt{\lambda C})^{{\beta}-1+\delta}\int_{0}^{\infty}s^{m-1+{\beta}-1+\delta} e^{-\sqrt{\lambda C}[s+1/s]}h(\sqrt{C/\lambda}s)\,\dd s.
\end{align*}

Recalling that $h(t)=e^{g(t)}$, we have that
\begin{align}\label{zz2}
\begin{split}
&\int_{0}^{\eps}s^{m-1} \Gamma({\beta},\sqrt{\lambda C}s)e^{-\sqrt{\lambda C}/s}h(\sqrt{C/\lambda}s)\,\dd s\\
&\quad\le\Gamma({\beta})\int_{0}^{\eps}s^{m-1} e^{-\sqrt{\lambda/C} s^{-1}[C-\sqrt{C/\lambda}sg(\sqrt{C/\lambda}s)]}\,\dd s.
\end{split}
\end{align}
Now, by the assumption in \eqref{log}, we have that $tg(t)\to0$ as $t\to0+$. Therefore, we may take $\lambda$ sufficiently large so that
\begin{align}\label{zz3}
C-\sqrt{C/\lambda}sg(\sqrt{C/\lambda}s)
>C/2>0\quad\text{for all }s\in[0,\eps].
\end{align}
Therefore, combining \eqref{zz2} and \eqref{zz3} yields
\begin{align*}
\int_{0}^{\eps}s^{m-1} \Gamma({\beta},\sqrt{\lambda C}s)e^{-\sqrt{\lambda C}/s}h(\sqrt{C/\lambda}s)\,\dd s
&\le\Gamma({\beta})\int_{0}^{\eps}s^{m-1} e^{-\frac{1}{2}\sqrt{\lambda C} s^{-1}}\,\dd s\\
&\le\Gamma({\beta})e^{-\frac{1}{2}\sqrt{\lambda C} \eps^{-1}}\int_{0}^{\eps}s^{m-1} \,\dd s.
\end{align*}

Therefore,
\begin{align}\label{terms}
\begin{split}
&\frac{\int_{0}^{\infty}mt^{m-1} S_{{\s}}(t)F(t)\,\dd t}{\int_{0}^{\infty}f_{{\s}}(t)F(t)\,\dd t}\\
&\le\frac{m}{\Gamma({\beta})}\Big(\frac{C}{\lambda}\Big)^{m/2}(\sqrt{\lambda C})^{{\beta}-1+\delta}\frac{\int_{0}^{\infty}s^{m-1+{\beta}-1+\delta} e^{-\sqrt{\lambda C}[s+1/s]}h(\sqrt{C/\lambda}s)\,\dd s}{\int_{0}^{\infty}f_{{\s}}(t)F(t)\,\dd t}\\
&\quad+m\Big(\frac{C}{\lambda}\Big)^{m/2}\frac{e^{-\frac{1}{2}\sqrt{\lambda C} \eps^{-1}}\int_{0}^{\eps}s^{m-1} \,\dd s}{\int_{0}^{\infty}f_{{\s}}(t)F(t)\,\dd t}.
\end{split}
\end{align}
Working on the second term in the righthand side of \eqref{terms}, we multiply by $\lambda^{m/2}$, change variables $s=\sqrt{\lambda/C}t$, and use \eqref{gamma} to obtain
\begin{align}\label{nn}
\begin{split}
&\frac{mC^{m/2}e^{-\frac{1}{2}\sqrt{\lambda C} \eps^{-1}}\int_{0}^{\eps}s^{m-1} \,\dd s}{\int_{0}^{\infty}f_{{\s}}(t)F(t)\,\dd t}
=\lambda^{-{\beta}}\Big(\frac{\lambda}{C}\Big)^{{\beta}/2}\frac{mC^{m/2}e^{-\frac{1}{2}\sqrt{\lambda C} \eps^{-1}}\int_{0}^{\eps}s^{m-1} \,\dd s}{\int_{0}^{\infty}s^{{\beta}-1}e^{-\sqrt{\lambda C}[s+1/s]}h(\sqrt{C/\lambda}s)\,\dd s}
\end{split}
\end{align}
Since $\eps\in(0,1/6)$ and $\max\{\delta_{1},\delta_{2}\}<3<\tfrac{1}{2}\eps^{-1}$ (recall the definitions of $\delta_{1}$ and $\delta_{2}$ in \eqref{delta1} and \eqref{delta2}), it follows from \eqref{claim3} and \eqref{nn} that
\begin{align}\label{easy}
\lim_{\lambda\to\infty}mC^{m/2}\frac{e^{-\frac{1}{2}\sqrt{\lambda C} \eps^{-1}}\int_{0}^{\eps}s^{m-1} \,\dd s}{\int_{0}^{\infty}f_{{\s}}(t)F(t)\,\dd t}=0.
\end{align}

Moving to the first term in the righthand side of \eqref{terms}, we change variables $s=\sqrt{\lambda/C}t$, recall \eqref{gamma}, and use Lemma~\ref{pq} to obtain
\begin{align}\label{lessroom}
\begin{split}
&\frac{m}{\Gamma({\beta})}\Big(\frac{C}{\lambda}\Big)^{m/2}(\sqrt{\lambda C})^{{\beta}-1+\delta}\frac{\int_{0}^{\infty}s^{m-1+{\beta}-1+\delta} e^{-\sqrt{\lambda C}[s+1/s]}h(\sqrt{C/\lambda}s)\,\dd s}{\int_{0}^{\infty}f_{{\s}}(t)F(t)\,\dd t}\\
&=m\Big(\frac{C}{\lambda}\Big)^{m/2}(\sqrt{\lambda C})^{{\beta}-1+\delta}\lambda^{-{\beta}}\Big(\frac{\lambda}{C}\Big)^{{\beta}/2}\frac{\int_{0}^{\infty}s^{m-1+{\beta}-1+\delta} e^{-\sqrt{\lambda C}[s+1/s]}h(\sqrt{C/\lambda}s)\,\dd s}{\int_{0}^{\infty}s^{{\beta}-1}e^{-\sqrt{\lambda C}[s+1/s]}h(\sqrt{C/\lambda}s)\,\dd s}\\
&\sim m\Big(\frac{C}{\lambda}\Big)^{m/2}(\sqrt{\lambda C})^{{\beta}-1+\delta}\lambda^{-{\beta}}\Big(\frac{\lambda}{C}\Big)^{{\beta}/2}
=mC^{(m-1+\delta)/2}\lambda^{-(m+1-\delta)/2}\quad\text{as }\lambda\to\infty.
\end{split}
\end{align}
Since $\delta\in(0,1)$, combining \eqref{terms} with \eqref{easy} and \eqref{lessroom} completes the proof.
\end{proof}

\section{Examples}\label{examples}

If the distribution of $\tau$ has the following short time behavior,
\begin{align}\label{log2}
\lim_{t\to0+}t\ln\P(\textcolor{black}{\tau\le t})=-C<0,
\end{align}
then Theorem~\ref{main} reveals the behavior of the $m$th conditional moment for every $m\ge1$,
\begin{align*}
\E[\tau^{m}\,|\,\tau<\s]
\sim\left(\frac{C}{\lambda}\right)^{m/2}\quad\text{as }\lambda\to\infty,
\end{align*}
where $\s$ is an independent Gamma random variable with shape $\beta>0$ and rate $\lambda>0$  (if ${\beta}=1$, then $\s$ is exponential). 

In this section, we show that \eqref{log2} is remarkably universal for FPTs of diffusion. Further, we identify the constant $C$ as
\begin{align}\label{C2}
C=\frac{L^{2}}{4D}>0,
\end{align}
where $D>0$ is a characteristic diffusion coefficient and $L>0$ is a certain geodesic distance between the possible searcher starting locations and the target. Therefore, Theorem~\ref{main} implies that
\begin{align}\label{formula2}
\E[\tau^{m}\,|\,\tau<\s]
\sim\left(\frac{L}{2\sqrt{D\lambda}}\right)^{m}\quad\text{as }\lambda\to\infty.
\end{align}
Our approach in this section follows reference \cite{lawley2020uni}, in which the short time behavior in \eqref{log2} was studied in the context of fastest FPTs.


Let $\{X(t)\}_{t\ge0}$ be a ${{d}}$-dimensional diffusion process (i.e.\ the ``searcher'') on a manifold $M$. Let $\tau>0$ be the FPT to a ``target'' $U_{\text{T}}\subset M$,
\begin{align}\label{tau}
\tau
:=\inf\{t>0:X(t)\in U_{\text{T}}\}.
\end{align}
Note that $U_{\text{T}}$ could consist of multiple regions, and thus includes the case of what might be considered ``multiple targets.'' Let ${{\s}}>0$ be an independent gamma random variable (i.e.\ the ``inactivation time'') with rate $\lambda>0$ and shape ${\beta}>0$ (if ${\beta}=1$, then $\s$ is exponential).

Suppose the target $U_{\text{T}}$ is the closure of its interior, which precludes trivial cases such as the target being a single point. Suppose the initial distribution of $X$ has compact support $U_{0}\subset M$ that does not intersect the target,
\begin{align}\label{away}
U_{0}\cap U_{\text{T}}=\varnothing.
\end{align}
For example, the initial distribution could be a Dirac delta function at a single point,
\begin{align*}
X(0)=x_{0}=U_{0}\in M,\quad\text{if }x_{0}\notin U_{\text{T}}.
\end{align*}
As another example, the initial distribution could be uniform on $U_{0}$ if the closed set $U_{0}$ satisfies \eqref{away}. Note that \eqref{away} ensures that the searcher cannot start arbitrarily close to the target.


\subsection{Simple $d$-dimensional diffusion}\label{section pure}

Consider first the case of pure diffusion in $M=\R^{{d}}$ with diffusivity $D>0$. In \cite{lawley2020uni}, it was shown that the distribution of $\tau$ satisfies \eqref{log2} with $C$ given in \eqref{C2}, where $L$ is the shortest distance from the starting locations, $U_{0}$, to the target, $U_{\text{T}}$,
\begin{align}\label{dset}
L
=\inf_{x_{0}\in U_{0},{{x}}\in U_{\text{T}}}\deuc({x_{0}},{{x}})>0,
\end{align}
where $\deuc({x_{0}},{{x}})$ is the standard Euclidean length between two points in $\R^{d}$,
\begin{align}\label{euc}
\deuc({x_{0}},{{x}})
:=\|x_{0}-x\|,\quad x_{0},x\in\R^{d}.
\end{align}
Hence, Theorem~\ref{main} ensures that the conditional $m$th moment of $\tau$ satisfies \eqref{formula2} with the length $L$ in \eqref{dset}.

\subsection{Space-dependent drift and diffusivity}\label{dd}

Suppose the searcher follows the It\^{o} stochastic differential equation on $M=\R^{{d}}$,
\begin{align}\label{sde}
\begin{split}
\dd X
&=b(X)\,\dd t+\sqrt{2D}{\Sigma}(X)\,\dd W.
\end{split}
\end{align}
In \eqref{sde}, $b:\R^{{d}}\to\R^{{d}}$ is the space-dependent drift that describes a deterministic force on the searcher, ${\Sigma}:\R^{{d}}\to\R^{{{d}}\times n}$ is a dimensionless function that describes any space-dependence or anisotropy in the diffusion, $D>0$ is a characteristic diffusion coefficient, and $W(t)\in\R^{n}$ is a standard $n$-dimensional Brownian motion. Following \cite{lawley2020uni}, assume that $\R^{{d}}\backslash U_{\text{T}}$ is bounded and that $b$ and ${\Sigma}$ satisfy some mild assumptions ($b$ is uniformly bounded and uniformly Holder continuous and ${\Sigma}{\Sigma}^{T}$ is uniformly Holder continuous and its eigenvalues lie in a finite interval $(\alpha_{1},\alpha_{2})$ with $\alpha_{1}>0$).

Given a smooth path $\omega:[0,1]\to M$, define its length in the Riemannian metric defined by the inverse of the diffusivity matrix $a:={\Sigma}{\Sigma}^{T}$,
\begin{align}\label{ll}
l(\omega)
:=\int_{0}^{1}\sqrt{\dot{\omega}^{T}(s)a^{-1}(\omega(s))\dot{\omega}(s)}\,\dd s.
\end{align}
Define the geodesic distance between any two points in $\R^{d}$,
\begin{align}\label{drie}
\begin{split}
\drie({x_{0}},{{x}})
&:=\inf\{l(\omega):\omega(0)=x_{0},\,\omega(1)=x\},\quad x_{0},x\in\R^{d},
\end{split}
\end{align}
where the infimum is over all smooth paths $\omega:[0,1]\to M$ from $\omega(0)={x_{0}}$ to $\omega(1)={{x}}$. In words, $\drie(x_{0},x)$ is the length of the optimal path from $x_{0}$ to $x$, where paths incur a cost for traveling through regions of slow diffusion. Notice that $\drie$ does not depend on the drift term in \eqref{sde}. Also notice that $\drie$ reduces to the Euclidean length $\deuc$ in \eqref{euc} if $a$ is the identity matrix (i.e.\ for isotropic, spatially constant diffusion).

Using Varadhan's formula~\cite{varadhan1967}, it was shown in \cite{lawley2020uni} that the distribution of $\tau$ satisfies \eqref{log2} with $C$ given in \eqref{C2}, where $L$ is
\begin{align}\label{dset2}
L
=\inf_{x_{0}\in U_{0},{{x}}\in U_{\text{T}}}\drie({x_{0}},{{x}})>0.
\end{align}
Hence, Theorem~\ref{main} ensures that the conditional $m$th moment of $\tau$ satisfies \eqref{formula2} with the length $L$ in \eqref{dset2}.

Counterintuitively, this reveals that the conditional FPT moments are completely independent of the drift $b$ in \eqref{sde} in the fast inactivation limit. Furthermore, this reveals how conditional FPTs depend on heterogeneous diffusion. Indeed, it shows that searchers which reach the target before a fast inactivation time avoid regions of slow diffusivity.

\subsection{Diffusion on a manifold with reflecting obstacles}\label{manifold}

Let $M$ be a ${{d}}$-dimensional smooth Riemannian manifold. We are most interested in the case that $M$ is a set in $\R^{3}$ with smooth outer and inner boundaries, which could model the cytosolic space bounded by the cell membrane (the outer boundary) and containing reflecting organelle obstacles (the inner boundaries). See Figure~\ref{figschem} for a 2-dimensional illustration. Another example of interest is the case where $M$ is the 2-dimensional surface of a 3-dimensional spheroid, which could model processes on a membrane.

Let $\{X(t)\}_{t\ge0}$ be a diffusion process on $M$ described by its generator $\L$, which in each coordinate chart is a second order differential operator of the form
\begin{align*}
\L f
=D\sum_{i,j=1}^{n}\frac{\partial}{\partial x_{i}}\Big(a_{ij}(x)\frac{\partial f}{\partial x_{j}}\Big),
\end{align*}
where the matrix $a=\{a_{ij}\}_{i,j=1}^{n}$ satisfies mild conditions (in each chart, $a$ is symmetric, continuous, and its eigenvalues lie in a finite interval $(\alpha_{1},\alpha_{2})$ with $\alpha_{1}>0$). Assume the diffusion reflects from the boundary of $M$ (if $M$ has a boundary) and assume $M$ is connected and compact.

Using \cite{norris1997}, it was shown in \cite{lawley2020uni} that the distribution of $\tau$ satisfies \eqref{log2} with $C$ given in \eqref{C2}, where $L$ is given by \eqref{dset2}. Hence, Theorem~\ref{main} ensures that the conditional $m$th moment of $\tau$ satisfies \eqref{formula2} with the length $L$ in \eqref{dset2}.

Therefore, searchers which reach the target before a fast inactivation time follow the shortest path to the target while avoiding obstacles. Notice that the infimum in \eqref{drie} is over paths in $M$, and thus paths that intersect obstacles are ineligible. 

\subsection{Partially absorbing targets}\label{partially}

In the preceding examples, the FPT of interest was the first time the searcher reached the target (see \eqref{tau}). In the literature, such targets are called ``perfectly absorbing,'' because it is envisioned that the searcher is ``absorbed'' immediately upon contact with the target.

However, targets are often modeled as \emph{partially absorbing} \cite{grebenkov2006, erban07}. The point of this subsection is to show that a partially absorbing target has no effect on conditional FPTs. That is, conditional FPT moments for partially absorbing targets are identical to conditional FPT moments for perfectly absorbing targets for fast inactivation.

Mathematically, the FPT for a searcher to be absorbed at a partially absorbing target is defined as
\begin{align}\label{taukappa}
\tau_{\kappa}
:=\inf\{t>0:l(t)>\xi_{\kappa}\},
\end{align}
where $\xi_{\kappa}$ is an independent exponential random variable with rate $\kappa>0$ and $l(t)$ is the so-called local time of $X(t)$ on the target $U_{\text{T}}$ \cite{grebenkov2006} (note that $l(t)$ has units of time/length and $\kappa$ has units of length/time). Equivalently, the backward Kolmogorov equation that describes the time evolution of the survival probability,
\begin{align}\label{skappa}
S(x,t):=\P(\tau_{\kappa}>t\,|\,X(0)=x),
\end{align}
for a partially absorbing target has a Robin boundary condition on the target involving the parameter $\kappa>0$.

To briefly illustrate, suppose the searcher diffuses on the positive real line with a partially absorbing target at the origin. The survival probability \eqref{skappa} satisfies
\begin{align}\label{kappabvp}
\begin{split}
\tfrac{\partial}{\partial t}S
&=D\tfrac{\partial^{2}}{\partial x^{2}}S,\quad x>0,\;t>0,\\
D\tfrac{\partial}{\partial x}S
&=\kappa S,\quad x=0,
\end{split}
\end{align}
with initial condition, $S(x,0)=1$. The solution to \eqref{kappabvp} is
\begin{align}\label{Skappa}
S(x,t)
=1-\text{erfc}\big(\tfrac{L}{\sqrt{4D t}}\big)+e^{\frac{\kappa  (\kappa  t+L)}{D}} \text{erfc}\big(\tfrac{2 \kappa  t+L}{\sqrt{4D t}}\big).
\end{align}
If the target was perfectly absorbing, then the boundary condition in \eqref{kappabvp} is pure Dirichlet (i.e.\ $\kappa=\infty$), and the solution reduces to $S_{\infty}(x,t):=1-\text{erfc}(\frac{L}{\sqrt{4D t}})$. Using these formulas, a straightforward calculation shows that \cite{lawley2020uni}
\begin{align}\label{logkappa}
\lim_{t\to0+}t\ln(1-S(L,t))
=\lim_{t\to0+}t\ln(1-S_{\infty}(L,t))
=-\frac{L^{2}}{4D}<0,\quad\text{if }L>0.
\end{align}
Therefore, the conditional $m$th moment of $\tau_{\kappa}$ satisfies \eqref{formula2} where $L>0$ is merely the distance from the starting location to the target. This leading order behavior of $\tau_{\kappa}$ is independent of $\kappa$ and is, in particular, the same for a partially absorbing target ($\kappa\in(0,\infty)$) as for a perfectly absorbing target ($\kappa=\infty$). 

The simple one-dimensional result in \eqref{logkappa} was extended to much more general situations in \cite{lawley2020uni}. More precisely, consider pure diffusion in a smooth bounded domain in $\R^{d}$ where the target is any finite disjoint union of hyperspheres. Let $\tau_{\kappa}$ be the FPT for the searcher to be absorbed in the case that the target is partially absorbing (i.e.\ $\tau_{\kappa}$ as in \eqref{taukappa}), and let $\tau$ be the FPT for the searcher to be absorbed in the case that the target is perfectly absorbing (i.e.\ $\tau$ as in \eqref{tau}). In this case, Ref.~\cite{lawley2020uni} proved that
\begin{align*}
\lim_{t\to0+}t\ln\P(\textcolor{black}{\tau_{\kappa}\le t})
=\lim_{t\to0+}t\ln\P(\textcolor{black}{\tau\le t}).
\end{align*}
Therefore, Theorem~\ref{main} reveals that the conditional FPT moments are unaffected by a partially absorbing target ($\kappa<\infty$) compared to a perfectly absorbing target ($\kappa=\infty$) in the fast inactivation limit.

{\black In fact, for the one-dimensional problem in \eqref{kappabvp}-\eqref{Skappa}, some conditional FPT moments can be calculated explicitly if the inactivation time $\sigma$ is exponentially distributed with rate $\lambda>0$ (i.e.\ $\beta=1$). In this case, a direct calculation using \eqref{ratio0} and \eqref{Skappa} yields
\begin{align}\label{2exact}
\begin{split}
\E[\tau\,|\,\tau<{\s}]
&=\left(\frac{L^{2}}{D}\right)\frac{\overline{\kappa} +\sqrt{\overline{\lambda} }+1}{2 (\overline{\kappa}  \sqrt{\overline{\lambda} }+\overline{\lambda} )},\\
\E[\tau^{2}\,|\,\tau<{\s}]
&=\left(\frac{L^{2}}{D}\right)^{2}\frac{(2 \overline{\kappa} +3) \overline{\lambda} +(\overline{\kappa} +1) (\overline{\kappa} +3) \sqrt{\overline{\lambda} }+\overline{\kappa}  (\overline{\kappa} +1)+\overline{\lambda} ^{3/2}}{4 \overline{\lambda} ^{3/2} (\overline{\kappa} +\sqrt{\overline{\lambda} })^2},
\end{split}
\end{align}
where we have defined the dimensionless inactivation rate and target reactivity,
\begin{align*}
\overline{\lambda}
:=\frac{\lambda L^{2}}{D},\quad
\overline{\kappa}
:=\frac{\kappa L}{D}.
\end{align*}
From \eqref{2exact}, we also obtain the coefficient of variation,
\begin{align}\label{cvexact}
\frac{\sqrt{\textup{Variance}[\tau\,|\,\tau<\s]}}{\E[\tau\,|\,\tau<\s]}
=\frac{\sqrt{2 (\overline{\kappa} +1) \sqrt{\overline{\lambda} }+\overline{\kappa}  (\overline{\kappa} +1)+\overline{\lambda} }}{{\overline{\lambda} }^{1/4} (\overline{\kappa} +\sqrt{\overline{\lambda} }+1)}.
\end{align}
Note that taking $\lambda\to\infty$ in \eqref{2exact} and \eqref{cvexact} agrees with Theorem~\ref{main} and Corollary~\ref{cor}. Note also that the formulas in \eqref{2exact}-\eqref{cvexact} diverge in the limit $\lambda\to0+$, which reflects the fact that the unconditioned mean FPT is infinite for this problem on a semi-infinite spatial domain.
}

\section{Discussion}

\subsection{Confirmation of conjecture of Reference \cite{ma2020}}

Consider the case of pure diffusion in a domain in $\R^{3}$ with smooth outer and inner boundaries (obstacles) as in Figure~\ref{figschem}a, which fits into the framework of section~\ref{manifold} above. Let $\tau_{\phys}$ denote the FPT for the searcher to reach the target, where the subscript emphasizes that this corresponds to the physiological case in which the searcher diffuses in the presence of reflecting obstacles within the domain. Let $\tau_{\no}$ denote the FPT for the searcher to reach the target in the case of no obstacles within the domain. That is, $\tau_{\no}$ corresponds to the same setup as $\tau_{\phys}$, except that we delete all the interior boundaries, see Figure~\ref{figschem}b.

In Reference \cite{ma2020}, it was conjectured that 
\begin{align}\label{conj}
\frac{\E[\tau_{\phys}\,|\,\tau_{\phys}<\s]}{\E[\tau_{\no}\,|\,\tau_{\no}<\s]}
\sim\frac{L_{\phys}}{L_{\no}}\quad\text{as }\lambda\to\infty,
\end{align}
where $L_{\phys}>0$ denotes the shortest distance from $U_{0}$ (the possible searcher starting locations) to $U_{\text{T}}$ (the target) that avoids the interior obstacles, and $L_{\no}\in(0,L_{\phys}]$ denotes the shortest distance from $U_{0}$ to $U_{\text{T}}$ regardless of interior obstacles. In \eqref{conj}, the inactivation time $\s$ is exponentially distributed with rate $\lambda>0$. Reference \cite{ma2020} obtained partial results to the effect of \eqref{conj} in the case when there are straight line paths from $U_{0}$ to $U_{\text{T}}$ and the principal curvatures of $\partial U_{\text{T}}$ satisfy certain constraints, but the general case was put forward as an open problem. We also note that Reference \cite{ma2020} proved \eqref{conj} in the case of a certain discretized diffusion model.

The results of the present work confirm the conjecture in \eqref{conj}. Indeed, combining Theorem~\ref{main} with the setting of section~\ref{manifold} above shows that if $\s$ is a gamma random variable with shape $\beta>0$ and rate $\lambda>0$, then for any moment $m\ge1$,
\begin{align}\label{conjs}
\begin{split}
\E[\tau_{\phys}^{m}\,|\,\tau_{\phys}<\s]
&\sim\left(\frac{L_{\phys}}{2\sqrt{D\lambda}}\right)^{m}\quad\text{as }\lambda\to\infty,\\
\E[\tau_{\no}^{m}\,|\,\tau_{\no}<\s]
&\sim\left(\frac{L_{\no}}{2\sqrt{D\lambda}}\right)^{m}\quad\text{as }\lambda\to\infty.
\end{split}
\end{align}
If we take $\beta=1$ (so that $\s$ is exponential with rate $\lambda$) and $m=1$, then \eqref{conjs} implies \eqref{conj}.

From a biological standpoint, the significance of \eqref{conj} is that if inactivation is fast (large $\lambda$), then the effect of a crowded  cytosolic space is reduced compared to an empty cytosolic space. That is, a diffusive signal which reaches the target before a fast inactivation time will take a nearly direct route, thereby reducing the effect of organelle barriers.

Our results show that this basic biological result is robust to how the heterogenous nature of the intracellular space is modeled. To explain, Reference \cite{ma2020} modeled the spatial heterogeneity of the cytosolic space by including hard reflecting barriers inside the domain. However, crowded intracellular geometry has also been modeled by an effective force field (i.e.\ a drift) that pushes searchers away from regions of dense obstacles. For example, Reference \cite{isaacson2011} used microscopic imaging to study how volume exclusion by chromatin affects the time it takes regulatory proteins to find binding sites, and the chromatin was modeled by an effective force field. As a third alternative, diffusion in the cytosolic space has also been modeled by diffusion with a space-dependent diffusion coefficient so that diffusion is slower in regions of dense obstacles \cite{cherstvy2013}.

Our results apply to each of these three modeling choices (reflecting obstacles, effective force field, or space-dependent diffusivity). Indeed, suppose $\tau_{\phys}$ denotes the FPT to reach the target in the case that the intracellular geometry is modeled by an effective force-field (rather than by reflecting obstacles). Suppose $\tau_{\no}$ follows the same setup as $\tau_{\phys}$ but with zero force-field inside the cell. This scenario fits into the framework of section~\ref{dd} above (with an appropriate choice of $b$ and ${\Sigma}$ in \eqref{sde} to effectively prevent the searcher from escaping the cell). Hence, we can apply Theorem~\ref{main} to conclude that \eqref{conjs} and \eqref{conj} hold with $L_{\phys}=L_{\no}$ (since the force does not effect the Riemannian metric used in the length \eqref{ll}).

Alternatively, suppose $\tau_{\phys}$ denotes the FPT to reach the target in the case that the intracellular geometry is modeled by a space-dependent diffusion coefficient, and suppose $\tau_{\no}$ follows the same setup as $\tau_{\phys}$ but with constant diffusion inside the cell. Again, applying Theorem~\ref{main} and the results of section~\ref{dd} imply that \eqref{conjs} and \eqref{conj} hold with $L_{\phys}\ge L_{\no}$, where these two lengths may differ because $L_{\phys}$ avoids regions of slow diffusion (see \eqref{ll}).

\subsection{\textcolor{black}{Other initial conditions}}
{\black

Our analysis assumed that the searcher cannot start arbitrarily close to the target. More precisely, we assumed that the unconditioned FPT distribution has the short-time behavior,
\begin{align}\label{log2d}
\lim_{t\to0+}t\ln\P(\textcolor{black}{\tau\le t})=-C<0,
\end{align}
which typically holds if the support of the initial searcher distribution does not intersect the target (see \eqref{away}). What is the conditional FPT distribution if the searcher can start arbitrarily close to the target?

If the searcher can start arbitrarily close to the target (meaning \eqref{away} is violated), then \eqref{log2d} may not hold and the distribution may instead decay algebraically,
\begin{align}\label{algebraic}
\P(\textcolor{black}{\tau\le t})
\sim At^{p},\quad\text{as }t\to0+,\quad\text{for some $A>0$, $p>0$.}
\end{align}
 For example, if the searchers are initially uniformly distributed in a bounded spatial domain, then it is typically the case that \cite{lawley2020comp, grebenkov2020}
\begin{align}\label{uniform}
\begin{split}
\P(\textcolor{black}{\tau\le t})
\sim\begin{cases}
A_{0}t^{1/2} & \text{if the target is perfectly absorbing},\\
A_{1}t & \text{if the target is partially absorbing},
\end{cases}\quad\text{as }t\to0+,
\end{split}
\end{align}
for constants $A_{0},A_{1}>0$ (see section~\ref{partially} for a description of perfectly absorbing and partially absorbing targets). 

If the unconditioned FPT distribution obeys \eqref{algebraic}, then we can determine the behavior of the conditional FPT moments in the fast inactivation limit. Indeed, Lemma~\ref{convenient} and a direct application of Laplace's method yields the following result. For simplicity, we assume $\sigma$ is exponentially distributed.

\begin{proposition}\label{str}
Let $\tau$ be any random variable satisfying \eqref{algebraic}. Suppose ${{\s}}>0$ is an independent exponential random variable with rate $\lambda>0$. If $m>0$, then
\begin{align*}
\E[\tau^{m}\,|\,\tau<{{\s}}]
\sim \frac{\Gamma(m+p)}{\Gamma(p)}\frac{1}{\lambda^{m}}
\quad\text{as }\lambda\to\infty.
\end{align*}
\end{proposition}
Combining Proposition~\ref{str} with \eqref{uniform} thus gives the following typical behavior of the conditional FPTs for a uniform initial searcher distribution,
\begin{align*}
\E[\tau^{m}\,|\,\tau<{{\s}}]
\sim\begin{cases}
\frac{\Gamma(m+1/2)}{\Gamma(1/2)}\frac{1}{\lambda^{m}} & \text{if the target is perfectly absorbing},\\
\Gamma(m+1)\frac{1}{\lambda^{m}} & \text{if the target is partially absorbing},
\end{cases}\quad\text{as }t\to0+.
\end{align*}

Another initial distribution which allows the searcher to start arbitrarily close to the target is the so-called quasi-stationary distribution \cite{meleard12}. If the searcher starts in its quasi-stationary distribution, which is generally given by the normalized principal eigenfunction corresponding to the Fokker-Planck equation describing searcher motion \cite{handy2019}, then the FPT is exactly exponentially distributed with rate given by the appropriately scaled principal eigenvalue $\lambda_{0}>0$. Therefore, Proposition~\ref{str} applies in this case. In fact, since both $\tau$ and $\sigma$ are exponential, we can make a much stronger statement and obtain the full conditional distribution of $\tau$ for any $\lambda>0$,
\begin{align*}
\P(\textcolor{black}{\tau\le t}\,|\,\tau<\sigma)
=1-e^{-(\lambda+\lambda_{0})t},\quad t\ge0.
\end{align*}

Notice that the conditional moments in Proposition~\ref{str} decay faster than the conditional moments in Theorem~\ref{main}. This is not surprising, since Proposition~\ref{str} corresponds to initial distributions which allow the searcher to start infinitesimally close to the target. For so-called extreme FPTs or fastest FPTs, it was similarly shown that the initial searcher distribution strongly affects the FPT distribution in the limit of many searchers (see \cite{weiss1983} and the more recent works \cite{lawley2020comp} and \cite{grebenkov2020}).
}

\subsection{Relations to other prior work}

As described in the Introduction, the finite lifetime of diffusing agents characterizes many biological and chemical processes. In the physics and chemistry literatures, such searchers are called ``mortal'' or ``evanescent'' \cite{abad2010, abad2012, abad2013, yuste2013, meerson2015b, meerson2015, grebenkov2017}.

Motivated by sperm cells searching for an egg, Meerson and Redner studied the competition between searcher mortality (inactivation), redundancy (many searchers), and diversity (different searchers) \cite{meerson2015}. This prior work derived several asymptotic results by exploiting some exact formulas that are available for one-dimensional, purely diffusive searchers with a perfectly absorbing target \cite{meerson2015}. In fact, for such a searcher starting at $x=L$ and diffusing with diffusivity $D>0$ on the positive real line with a perfectly absorbing target at the origin, the authors found that 
\begin{align*}
\E[\tau\,|\,\tau<\s]
=\frac{L}{2\sqrt{D\lambda}}\quad\text{for all }\lambda>0,
\end{align*}
if the inactivation time $\sigma$ is exponential with rate $\lambda>0$ \cite{meerson2015}. That is, the general asymptotic formula that we derived for $\lambda\gg1$ is exact for any choice of $\lambda>0$ in this simplified problem.

In another interesting study of purely diffusive searchers with exponential lifetimes, Grebenkov and Rupprecht derived upper and lower bounds on certain distributions and statistics \cite{grebenkov2017}. In contrast to the present work, Reference \cite{grebenkov2017} considered the \emph{minimum} of the FPT to reach the target and the inactivation time,
\begin{align}\label{min}
\min\{\tau,\s\},
\end{align}
rather than the conditional FPT. In the limit of fast inactivation, these authors derived the approximation $\E[\min\{\tau,\s\}]\approx1/\lambda$, which is intuitive since $\E[\s]=1/\lambda$ and $\P(\tau<\s)\to0$ as $\lambda\to\infty$.

\subsection{Conclusion}

In this paper, we determined all the moments of the time it takes a diffusive searcher to find a target, conditioned that the searcher finds the target before a fast inactivation time. This moment formula holds under very general conditions on the diffusive searcher dynamics, the target, and the spatial domain. These results prove in significant generality that if inactivation is fast, then the conditional FPT compared to the unconditioned FPT is much faster, much less affected by obstacles, and much less variable.

These effects stem from fast inactivation filtering out searchers which do not take a direct path to the target. Similar effects arise from so-called fastest or extreme FPTs \cite{lawley2020uni, lawley2020esp1, lawley2020esp4, basnayake2019, schuss2019, coombs2019, redner2019, sokolov2019, rusakov2019, martyushev2019, tamm2019, basnayake2019c}. In both cases, the complexity of cell biology is modifying traditional notions of diffusive timescales.

\bibliography{library.bib}
\bibliographystyle{siam}

\end{document}